\RequirePackage{fix-cm}
\documentclass[smallextended]{svjour3}       
\smartqed  
\usepackage{graphicx}
\usepackage[utf8]{inputenc}
\usepackage[english]{babel}
\usepackage{bbm}

\usepackage[nottoc]{tocbibind}
\usepackage{mathrsfs,amsfonts,amssymb,amsmath}

\usepackage{amsthm}
\usepackage{enumerate}
\usepackage{graphicx,cite}
\usepackage{romannum}
\usepackage{hyperref}
\usepackage{authblk}
\usepackage{graphicx}

\allowdisplaybreaks

\hypersetup{colorlinks=true, linkcolor=blue, citecolor=red}

\numberwithin{equation}{section}

\newtheorem{notation}[theorem]{Notation}

\newtheorem*{Example}{Example}

\usepackage[square,numbers]{natbib}
\begin{document}
\pagenumbering{arabic}
\setlength{\belowdisplayskip}{0pt}

\title{Estimates of constants in the limit theorems for chaotic dynamical systems}


\author{Leonid A. Bunimovich  \and
        Yaofeng Su
}


\institute{Leonid A. Bunimovich \at
              School of Mathematics, Georgia Institute of Technology,  Atlanta, USA\\
             \email{leonid.bunimovich@math.gatech.edu}           
           \and
            Yaofeng Su  \at
              School of Mathematics, Georgia Institute of Technology,  Atlanta, USA\\
              \email{yaofeng.su@math.gatech.edu} 
}

\date{Received: date / Accepted: date}

\maketitle

\begin{abstract}
In a vast area of probabilistic limit theorems for dynamical systems with chaotic behaviors always only functional form (exponential, power, etc) of the asymptotic laws and of convergence rates were studied. However, for basically all applications, e.g., for computer simulations, development of algorithms to study chaotic dynamical systems numerically, as well as for design and analysis of real (e.g., in physics) experiments, the exact values (or at least estimates) of constants (parameters) of the functions, which appear in the asymptotic laws and rates of convergence, are of primary interest. In this paper we provide such estimates of constants (parameters) in the central limit theorem, large deviations principle, law of large numbers and the rate of correlations decay for strongly chaotic dynamical systems.

\end{abstract}

\tableofcontents

\section{Introduction}\ \par
Dynamical systems with an irregular (``complex") behavior (loosely called chaotic systems) were discovered in the last century. Their evolution (in ``time") looks similar to the evolution of random processes. Since the remarkable Kolmogorov's paper \cite{K} which built a ``bridge" between the worlds of deterministic (dynamical) systems and random systems (stochastic processes),
a natural question was raised whether some classes of chaotic dynamical systems satisfy limit theorems known for stochastic processes.
One should mention that ergodicity (which is equivalent to the strong law of large numbers) was established for various dynamical systems. However, non-chaotic systems can be ergodic too. Perhaps, the most famous (and somewhat simple) example of ergodic, but non-chaotic, system is the irrational rotation of a circle. Therefore one should generally impose on a dynamical system some stronger, than ergodicity, statistical properties. 
Starting with the second half of the last century studies of limit theorems for chaotic dynamical systems became one of the most active areas in the theory of dynamical systems. The results obtained in this area allowed for many breakthroughs in various applications of dynamical systems. One of the most explored areas was related to the problem of the existence of various transport coefficients in nonequilibrium statistical mechanics, which requires estimations of the rates of correlations decay for some meaningful observables.
The results in this area deal with establishing a functional form of the limits of various averages along the orbits of dynamical systems,  or of the rates of convergence to these limits. For instance, for existence of transport coefficients it is enough to prove that correlations of some specific  functions (observables) on the phase space of a system in question are integrable over an infinite time interval, i.e., the correlations decay sufficiently fast. 
All estimates in these limit theorems contain various constants, e.g., the exponential estimate looks like $a\cdot \exp(b\cdot t)$. However, what are the values of $a$ and $b$ in this expression? Such questions are of importance for applications of the dynamical systems theory in any areas of science and technology, particularly for numerical studies and simulations of chaotic dynamical systems, for developing the corresponding algorithms, and even for building relevant devices for experimental studies of chaotic systems.
In the present paper we obtain estimates of constants appearing in main limit theorems of probability theory like the central limit theorem, the large deviation principle, the law of large numbers and the rate of correlations decay. 

The organization of the paper: in section 2 we first introduce the definitions, notations, and main theorems. In section 3, we prove these theorems. We conclude this paper with applications in section 4.


\section{Definitions and Notations}\label{positive}

\begin{definition}[Topological Markov shifts]\label{shift}
    The phase space $\mathbb{A}$ of a topological Markov shift with finite states space $S$ is defined as \[\mathbb{A}:=\{(x_0, x_1, x_2 \cdots): x_i \in S, \quad \#S< \infty\}.\]

The dynamics on $\mathbb{A}$ is the left shift $\sigma: \mathbb{A}\to \mathbb{A}: \sigma(x_0, x_1, \cdots) = (x_1, x_2, \cdots)$.
Let there is $\theta \in (0,1)$, such that for any two points $x,y\in \mathbb{A}$, the distance $d(x,y)$ between them equals \[d(x,y):=\theta^{\inf\{n\ge 0:\text{ } x_n\neq y_n\}}.\] 

Consider a potential $\phi_p: \mathbb{A}\to \mathbb{R}$, which is a Lipshitz function. The last means that there is $||\phi_p||\ge 1$, such that \[||\phi_p||_{\infty}\le ||\phi_p||, \quad |\phi_p(x)-\phi_p(y)|\le ||\phi_p||d(x,y).\]  

We also suppose that $\phi_p$ is normalized, i.e., it satisfies the relation \[\sum_{a \in S}e^{\phi_p(ax)}=1 \text{ for any } x \in \mathbb{A}.\]

    An equilibrium state is a probability measure $\mu_{\phi_p}$, which is shift-invariant (i.e. $\mu_{\phi_p}\circ \sigma^{-1}=\mu_{\phi_p}$). Denote the expectation with respect to $\mu_{\phi_p}$ by $\mathbb{E}$.

    Next, a transfer operator is defined as \[P(\phi)(x)=\sum_{a \in S} e^{\phi_p(ax)}\phi(ax)\]

    In what follows, we will always denote an observable for the dynamical system $(\mathbb{A}, \sigma, \mu_{\phi_p})$  by $\phi: \mathbb{A} \to \mathbb{R}$. We also assume that all observables under consideration are Lipschitz, i.e., there exists $||\phi||\ge 1$, such that \[||\phi||_{\infty}\le ||\phi||, \quad|\phi(x)-\phi(y)|\le ||\phi||d(x,y).\]

\end{definition}

\begin{remark}
    For uniformly hyperbolic (e.g., the Axiom A) systems, $\theta$ is the rate of expansion along unstable manifolds, and $||\phi_p||$ is the constant appearing in the expression of the distortion (ratio of the expansion rates) along unstable manifolds.
\end{remark}
\begin{remark}
    To estimate the constants appearing in the exponential mixing rates we will use the maximal coupling techniques \cite{max}.
\end{remark}

The main results of the present paper are listed below.

\begin{theorem}[Estimates of an error in the rate of correlations decay] \label{correlation}\ \par
Consider the topological dynamical system in Definition \ref{shift}. We have
\[||P^n(\phi-\mathbb{E}\phi)||_{\infty} \le \frac{(\sqrt{z_0})^{-n}}{(\log\sqrt{z_0})(1-a-\epsilon)}||\phi||,\]

where $\epsilon \in (0, \min\{1-a, 1-\theta\})$, $a=1-\exp(-||\phi_p||(1-\theta)^{-1})$ and 
\[z_0 \in \Big(1, \min\Big\{\Big(\frac{a+\epsilon/2}{a}\Big)^{U}, \frac{\theta+\epsilon}{\theta}, \frac{1-\epsilon}{\theta}, e^2\Big\}\Big), \quad U:=\frac{\log (1-\epsilon)}{\log [\theta \epsilon(1-\epsilon-\theta)]-\log (2||\phi_p||)}\] 
\end{theorem}

\begin{theorem}[Estimates of an error in the CLT]\label{clt}\ \par
Consider the topological dynamical system in Definition \ref{shift}. Suppose that the observable $\phi$ has zero mean, and let $\Sigma^2:=\mathbb{E}\phi^2+2\sum_{i\ge 1}\mathbb{E}\phi \phi^i$. Then for any $t\in \mathbb{R}$ 
    \[\Big|\mathbb{E}\exp\Big(it \frac{\sum_{0 \le i \le n-1}\phi\circ \sigma^i}{\sqrt{n}}\Big)-\exp\Big(-\frac{t^2\Sigma^2}{2}\Big)\Big|\le C_{t,n}, \]
    where\begin{align*}C_{t,n}=&6922\Big(\frac{|t|^4||\phi||^4}{n^{0.1}(1-z_0^{-0.5})^4(1-a-\epsilon)^4(\ln \sqrt{z_0})^4}+\frac{t^2z_0^{-0.5n}||\phi||^2}{(1-z_0^{-0.5})(1-a-\epsilon)\ln \sqrt{z_0}}\Big)\\
    &+64\Big(\frac{t^2n^3||\phi||^2e^{2||\phi_p||\theta(1-\theta)^{-1}}||\phi_p||^2(\sqrt{z_0})^{-n^{0.2}}}{(1-\theta)^2(1-a-\epsilon)\ln \sqrt{z_0}}\Big),
    \end{align*} and $a, \epsilon, z_0$ are the ones in Theorem \ref{correlation}.
\end{theorem}

\begin{theorem}[Estimates of an error in the large deviations principle] \label{deviation}\ \par
Consider the topological dynamical system in Definition \ref{shift}.
Suppose again that $\phi$ has zero mean. Then for any $u>0$
    \begin{align*}
 &\mu_{\phi_p}\Big(\Big|\frac{\sum_{i \le n-1}\phi \circ \sigma^i}{n}\Big|\ge u\Big) \\
    &\le 2\exp \Big(\frac{u(1-z_0^{-0.5})(1-a-\epsilon)(\ln \sqrt{z_0})}{36||\phi||z_0^{0.5}}\Big)\exp\Big(-\frac{u^2n (1-z_0^{-0.5})^2(1-a-\epsilon)^2(\ln \sqrt{z_0})^2}{72||\phi||^2}\Big),
    \end{align*}
    where $a, \epsilon, z_0$ are the ones in Theorem \ref{correlation}.
\end{theorem}
\begin{remark}
    Note that a stronger form of the large deviation principle is proved in \cite{mldp}.
\end{remark}

\begin{theorem}[Estimates of an error in the law of large numbers]\label{largenumber}\ \par
Consider the topological dynamical system in Definition \ref{shift}. Suppose that $\phi$ has zero mean. Then there exists a sufficiently small $\delta\in (0,0.5) $, such that for a.e. $x \in \mathbb{A}$, any $n \ge N_{x, \delta},$\[\Big|\frac{\sum_{i \le n}\phi \circ \sigma^i(x)}{n}\Big|\le n^{-(0.5-\delta)},\]
where $N_{x,\delta}:= \sum_{N\ge 0} \mathbbm{1}_{\bigcup_{n\ge N}|\sum_{i \le n-1}\phi \circ \sigma^i(x)|\ge n^{0.5+\delta}}$ satisfies the relation $||N_{x,\delta}||_1$ \[\le 2\exp \Big(\frac{(1-z_0^{-0.5})(1-a-\epsilon)(\ln \sqrt{z_0})}{36||\phi||z_0^{0.5}}\Big)\sum_{m\ge 1}m\exp\Big(-\frac{m^{2\delta} (1-z_0^{-0.5})^2(1-a-\epsilon)^2(\ln \sqrt{z_0})^2}{72||\phi||^2}\Big),\]
 and  $a, \epsilon, z_0$ are the ones in Theorem \ref{correlation}.
\end{theorem}

\section{Proofs}
\begin{proof}[A proof of Theorem \ref{correlation}]
Consider  a topological Markov shift in Definition \ref{shift}.  Recall now  the Theorem 1 of \cite{max}, which says that 
\[|\int f \cdot P^n (\phi-\mathbb{E}\phi) d\mu_{\phi_p}| \le ||f||_1 \frac{||\phi||}{||\phi_p||} \sum_{k=0}^{n}||\phi_p||\theta^k\gamma_{n-k}^{*},\]
therefore, 
\[||P^n (\phi-\mathbb{E}\phi)||_{\infty} \le ||\phi||\sum_{k=0}^{n}\theta^k\gamma_{n-k}^{*},\]
where $\gamma_{k}^{*}:=\mathbb{P}(S_k=0)$, $(S_n)_{n \ge 0}$ is a Markov chain with values in $\mathbb{N}\bigcup \{0\}$, which starts from $0$, and has transition probabilities $p_{k, k+1}=1-\gamma_k$, $p_{k,0}=\gamma_k$, and $\gamma_k:=1-\exp(-||\phi_p||\theta^k)$.

 We compute now $\gamma_k^{*}$. It follows from (A.3)-(A.5) of \cite{max} that 
 \[\sum_{k\ge 1}\mathbb{P}(\tau=k)<1,\] and
\[\sum_{k}\mathbb{P}(S_k=0)z_0^k=\frac{1}{1-\sum_{k\ge 1}z_0^k\mathbb{P}(\tau=k)} \text{ for some }z_0>1,\]
where \[\mathbb{P}(\tau=k)=\mathbb{P}(\inf\{n>0: S_n=0\}=k)= \prod_{i=0}^{k-2}p_{i,i+1} p_{k-1,0}\le ||\phi_p||\theta^{k-1},\]
\[\mathbb{P}(\tau=k)=\prod_{i=0}^{k-2}p_{i,i+1} p_{k-1,0} =\prod_{i=0}^{k-2}p_{i,i+1}-\prod_{i=0}^{k-1}p_{i,i+1}.\]

Choose now $a, z_0, U$ and $\epsilon$ as in the formulation of the Theorem \ref{correlation}, let $N=U^{-1}$. Then
\begin{align*}
    \sum_{i\ge 1}\mathbb{P}(\tau=i)z_0^i&=\sum_{i\ge N}\mathbb{P}(\tau=i)z_0^i+\sum_{i< N}\mathbb{P}(\tau=i)z_0^i\\
    &\le \sum_{i \ge N}\frac{||\phi_p||}{\theta}(z_0\theta)^i+\sum_{i <N}\mathbb{P}(\tau=i)z_0^i\\
    &\le \sum_{i \ge N}\frac{||\phi_p||}{\theta}(z_0\theta)^i+\sum_{i\ge 1}\mathbb{P}(\tau=i)z_0^N\\
    &\le \frac{||\phi_p||}{\theta}\frac{(z_0\theta)^N}{1-z_0\theta}+\Big[1-\exp\Big(-\frac{||\phi_p||}{1-\theta}\Big)\Big]z_0^N\le a +\epsilon.
\end{align*}

Therefore $\mathbb{P}(S_k=0)z_0^k \le \frac{1}{1-a-\epsilon}$, and \[||P^n (\phi-\mathbb{E}\phi)||_{\infty} \le ||\phi||\sum_{k=0}^{n}\theta^k\gamma_{n-k}^{*}\le ||\phi||\frac{nz_0^{-n}}{1-a-\epsilon} \le ||\phi||\frac{(\sqrt{z_0})^{-n}}{(1-a-\epsilon)\log \sqrt{z_0}},\]
which implies the statement of Theorem \ref{correlation}.
\end{proof}

\begin{proof}[Proof of Theorem \ref{clt}]
We adopt the martingale methods used in \cite{Sutams}. (See also \cite{martingale}, where martingale methods were applied to analysis of partially hyperbolic dynamical systems). Let $\mathcal{B}$ be the $\sigma$-algebra of $\mathbb{A}$, $\mathcal{F}_n:=\sigma^{-n}\mathcal{B}$ and \[H_n \circ \sigma^n:=\mathbb{E}\Big[\sum_{i \le n-1}\phi \circ \sigma^i|\mathcal{F}_n\Big],\]
\[\psi_n\circ \sigma^n:=\phi \circ \sigma^n+H_n\circ \sigma^n -H_{n+1}\circ \sigma^{n+1}\]
Then $H_n=\sum_{i \le n-1}P^{n-i}\phi$, $(\psi_n \circ \sigma^n)_{n \ge 0}$ are reverse martingale differences and
\begin{equation}\label{1}
    \sum_{n \le N}\psi_n \circ \sigma^n=\sum_{n \le N}\phi \circ \sigma^n-H_{N+1}\circ \sigma^{N+1}.
\end{equation}

By applying Theorem \ref{correlation}, we get, for any $n\ge 1$, that \begin{equation}\label{2}
    ||P^n\phi||_{\infty}\le ||\phi||, \quad \sup_n||H_n||_{\infty}\le ||\phi||\frac{z_0^{-0.5}}{(1-z_0^{-0.5})(1-a-\epsilon)\ln \sqrt{z_0}},
\end{equation}\begin{equation}\label{3}
    ||\psi_n||_{\infty}\le ||\phi||_{\infty}+||H_n||_{\infty}+||H_{n+1}\circ \sigma||_{\infty}\le \frac{3||\phi||}{(1-z_0^{-0.5})(1-a-\epsilon)\ln \sqrt{z_0}},
\end{equation}

\[\Sigma^2=\lim_{N\to \infty}\frac{\mathbb{E}(\sum_{n \le N}\phi \circ \sigma^n)^2}{N}=\lim_{N \to \infty}\frac{\sum_{n \le 
N}\mathbb{E}\psi_n^2\circ \sigma^n}{N},\]
\[\Big|\frac{\mathbb{E}(\sum_{n \le N}\phi \circ \sigma^n)^2}{N}-\frac{\sum_{n \le 
N}\mathbb{E}\psi_n^2\circ \sigma^n}{N}\Big|=\frac{\mathbb{E}H^2_{N+1}}{N}\le \frac{||\phi||^2z_0^{-1}}{N(1-z_0^{-0.5})^2(1-a-\epsilon)^2(\ln \sqrt{z_0})^2}.\]

By making use of Theorem \ref{correlation}, and of the argument in the proof of Corollary 3.10 in \cite{Sutams}, we obtain the following relations \[\Big|\frac{\mathbb{E}(\sum_{n \le N}\phi \circ \sigma^n)^2}{N}-\Sigma^2\Big| \le ||\phi||^2\frac{z_0^{-0.5}}{N(1-z_0^{-0.5})^2(1-a-\epsilon)\ln \sqrt{z_0}}+\frac{z_0^{-0.5(N+1)}||\phi||^2}{(1-z_0^{-0.5})(1-a-\epsilon)\ln \sqrt{z_0}},\]
\begin{equation}\label{4}
    \Big|\frac{\sum_{n \le N}\mathbb{E}\psi_n^2 \circ \sigma^n}{N}-\Sigma^2\Big| \le \frac{2||\phi||^2z_0^{-0.5}}{N(1-z_0^{-0.5})^2(1-a-\epsilon)^2(\ln \sqrt{z_0})^2}+\frac{z_0^{-0.5(N+1)}||\phi||^2}{(1-z_0^{-0.5})(1-a-\epsilon)\ln \sqrt{z_0}}.
\end{equation}

Next we address a regularity of $\psi_n$. One has
\begin{align*}
    &|(P^n\phi)(x_1)-(P^n\phi)(x_2)|=\big|\sum_{\sigma^ny_1=x_1} e^{\sum_{i \le n-1} \phi_p(\sigma^i y_1)}\phi(y_1)-\sum_{\sigma^ny_2=x_2} e^{\sum_{i \le n-1} \phi_p(\sigma^i y_2)}\phi(y_2)\big|\\
    &=\big|\sum e^{\sum_{i \le n-1} \phi_p(\sigma^i y_2)}\phi(y_1)[e^{\sum_{i \le n-1} \phi_p(\sigma^i y_1)-\phi_p(\sigma^i y_2)}-1]+\sum e^{\sum_{i \le n-1} \phi_p(\sigma^i y_2)}[\phi(y_1)-\phi(y_2)]\big|\\
     &\le |\sum e^{\sum_{i \le n-1} \phi_p(\sigma^i y_2)}\phi(y_1)[e^{\sum_{i \le n-1} \phi_p(\sigma^i y_1)-\phi_p(\sigma^i y_2)}-1]|+||\phi||\theta^n d(x_1,x_2)\\
     &\le ||\phi||[e^{||\phi_p||d(x_1,x_2)\theta(1-\theta)^{-1}}-1]+||\phi||\theta^n d(x_1,x_2)\\
     &\le ||\phi||[e^{||\phi_p||\theta(1-\theta)^{-1}}||\phi_p||\theta(1-\theta)^{-1}+\theta^n]d(x_1,x_2)\le 2\theta||\phi||e^{||\phi_p||\theta(1-\theta)^{-1}}||\phi_p||(1-\theta)^{-1}d(x_1,x_2).
\end{align*}

Therefore, $||P^n\phi||\le 2||\phi||e^{||\phi_p||\theta(1-\theta)^{-1}}||\phi_p||(1-\theta)^{-1}$, and

\[||H_n||\le 2n||\phi||e^{||\phi_p||\theta(1-\theta)^{-1}}||\phi_p||(1-\theta)^{-1},\]
\begin{equation}\label{5}
    ||\psi_n||\le ||\phi||+||H_n||+||H_{n+1}\circ \sigma||\le (4n+3)||\phi||e^{||\phi_p||\theta(1-\theta)^{-1}}||\phi_p||(1-\theta)^{-1}.
\end{equation}
\begin{notation}
    In order to simplify the proofs, we will write in what follows $f=O_C(g)$  instead of $|f|\le C |g|$.
\end{notation}
Next we estimate the following expression $$\Big|\mathbb{E}\exp\Big(it \frac{\sum_{0 \le i \le n-1}\phi\circ \sigma^i}{\sqrt{n}}\Big)-\exp\Big(-\frac{t^2\Sigma^2}{2}\Big)\Big|.$$ 

Consider the blocks $\{I_{i}: \mathbf{card}(I_i)=n^{0.8},  1\le i \le n^{0.2}\}$ of consecutive positive integers in $[0,n-1]$, and the blocks $\{I_{i, 1}, I_{i,2}: \mathbf{card}(I_{i,1})=n^{0.8}-n^{0.2}, \mathbf{card}(I_{i,2})=n^{0.2}\}$ of consecutive positive integers in each $I_{i}$, which have no gaps between the blocks.  Let  $X_{i}:= \sum_{k \in I_{i,1}} \psi_k \circ \sigma^k$, $\mathcal{F}_i=\sigma^{-in^{0.8}}\mathcal{B}$. Then, it follows from (\ref{1}) and (\ref{2}), that \begin{align*}
    \mathbb{E}\exp\Big(&it \frac{\sum_{0 \le i \le n-1}\phi\circ \sigma^i}{\sqrt{n}}\Big)=\mathbb{E}\exp\Big(it \frac{\sum_{0 \le i \le n-1}\psi_i\circ \sigma^i}{\sqrt{n}}\Big)+O_1\Big(\frac{|t|||\phi||z_0^{-0.5}}{n^{0.5}(1-z_0^{-0.5})(1-a-\epsilon)\ln \sqrt{z_0}}\Big)\\
    &=\exp\Big(-\frac{t^2\sum_{i \le n-1}\mathbb{E}\psi_i^2\circ \sigma^i}{2n}\Big)+ \mathbb{E}\exp\Big(it \frac{\sum_{0 \le i \le n-1}\psi_i\circ \sigma^i}{\sqrt{n}}\Big)\\
    &\quad -\exp\Big(-\frac{t^2\sum_{i \le n-1}\mathbb{E}\psi_i^2\circ \sigma^i}{2n}\Big)+O_1\Big(\frac{|t|||\phi||z_0^{-0.5}}{n^{0.5}(1-z_0^{-0.5})(1-a-\epsilon)\ln \sqrt{z_0}}\Big).
    \end{align*}
    
    By making use of (\ref{3}), we can now continue the estimates as 
    \begin{align*}
&=\exp\Big(-\frac{t^2\sum_{i \le n-1}\mathbb{E}\psi_i^2\circ \sigma^i}{2n}\Big)+\mathbb{E}\exp\Big(it\frac{\sum_{i=1}^{n^{0.2}}X_i}{\sqrt{n}}\Big)-\exp\Big(-\frac{t^2\sum_{i=1}^{n^{0.2}}\mathbb{E}X_i^2}{2n}\Big)\\
&\quad +O_3\Big(\frac{|t|n^{0.4}||\phi||}{n^{0.5}(1-z_0^{-0.5})(1-a-\epsilon)\ln \sqrt{z_0}}\Big)+O_{5}\Big(\frac{|t|^2n^{0.6}||\phi||^2}{n(1-z_0^{-0.5})^2(1-a-\epsilon)^2(\ln \sqrt{z_0})^2}\Big)\\
&\quad+O_1\Big(\frac{|t|||\phi||z_0^{-0.5}}{n^{0.5}(1-z_0^{-0.5})(1-a-\epsilon)\ln \sqrt{z_0}}\Big)\\
&=\exp\Big(-\frac{t^2\sum_{i \le n-1}\mathbb{E}\psi_i^2\circ \sigma^i}{2n}\Big)+\mathbb{E}\exp\Big(it\frac{\sum_{i=1}^{n^{0.2}}X_i}{\sqrt{n}}\Big)-\exp\Big(-\frac{t^2\sum_{i=1}^{n^{0.2}}\mathbb{E}X_i^2}{2n}\Big)\\
&\quad +O_9\Big(\frac{|t|^2||\phi||^2}{n^{0.1}(1-z_0^{-0.5})^2(1-a-\epsilon)^2(\ln \sqrt{z_0})^2}\Big).
\end{align*}

We will use the Taylor expansion, as in the proof of Corollary 3.1 of \cite{Sutams}, and apply the Burkholder–Davis–Gundy inequalities to $X_i$, we can further continue the estimates as
\begin{align*}
&=\exp\Big(-\frac{t^2\sum_{i \le n-1}\mathbb{E}\psi_i^2\circ \sigma^i}{2n}\Big)+\sum_{i=1}^{n^{0.2}}O_1\Big(\mathbb{E}\Big|\frac{tX_i}{n^{0.5}}\Big|^3+\Big|\frac{t^2\mathbb{E}X^2_i}{n}\Big|^2\Big)
\\
&\quad +\sum_{i=1}^{n^{0.2}}O_1\Big(\frac{t^2\mathbb{E}|\mathbb{E}[X^2_i-\mathbb{E}X^2_i|\mathcal{F}_i]|}{2n}\Big)+O_9\Big(\frac{|t|^2||\phi||^2}{n^{0.1}(1-z_0^{-0.5})^2(1-a-\epsilon)^2(\ln \sqrt{z_0})^2}\Big)\\
&=\exp\Big(-\frac{t^2\sum_{i \le n-1}\mathbb{E}\psi_i^2\circ \sigma^i}{2n}\Big)+\sum_{i=1}^{n^{0.2}}O_{64}\Big(\mathbb{E}(\sum_{k \in I_{i,1}}\psi_k^2\circ \sigma^k)^{1.5}\Big|\frac{t}{n^{0.5}}\Big|^3+\Big|\frac{t^2\mathbb{E}X^2_i}{n}\Big|^2\Big)
\\
&\quad +\sum_{i=1}^{n^{0.2}}O_1\Big(\frac{t^2\mathbb{E}|\mathbb{E}[X^2_i-\mathbb{E}X^2_i|\mathcal{F}_i]|}{2n}\Big)+O_9\Big(\frac{|t|^2||\phi||^2}{n^{0.1}(1-z_0^{-0.5})^2(1-a-\epsilon)^2(\ln \sqrt{z_0})^2}\Big).
\end{align*}

Now, from (\ref{3}) and (\ref{4}), we get that the previous expression equals 
\begin{align*}
&=\exp\Big(-\frac{t^2\sum_{i \le n-1}\mathbb{E}\psi_i^2\circ \sigma^i}{2n}\Big)+\sum_{i=1}^{n^{0.2}}O_{6912}\Big(\frac{|t|^4||\phi||^4}{n^{0.3}(1-z_0^{-0.5})^4(1-a-\epsilon)^4(\ln \sqrt{z_0})^4}\Big)\\
&\quad+\sum_{i=1}^{n^{0.2}}O_1\Big(\frac{t^2\mathbb{E}|\mathbb{E}[X^2_i-\mathbb{E}X^2_i|\mathcal{F}_i]|}{2n}\Big)+O_9\Big(\frac{|t|^2||\phi||^2}{n^{0.1}(1-z_0^{-0.5})^2(1-a-\epsilon)^2(\ln \sqrt{z_0})^2}\Big)\\
 &=\exp\Big(-\frac{t^2\sum_{i \le n-1}\mathbb{E}\psi_i^2\circ \sigma^i}{2n}\Big)+\sum_{i=1}^{n^{0.2}}O_1\Big(\frac{t^2\mathbb{E}|\mathbb{E}[X^2_i-\mathbb{E}X^2_i|\mathcal{F}_i]|}{n}\Big)\\
 &\quad +O_{6921}\Big(\frac{|t|^4||\phi||^4}{n^{0.1}(1-z_0^{-0.5})^4(1-a-\epsilon)^4(\ln \sqrt{z_0})^4}\Big)\\
  &=\exp\Big(-\frac{t^2\Sigma^2}{2}\Big)+O_{0.5}\Big(\frac{2t^2||\phi||^2z_0^{-0.5}}{n(1-z_0^{-0.5})^2(1-a-\epsilon)^2(\ln \sqrt{z_0})^2}+\frac{t^2z_0^{-0.5n}||\phi||^2}{(1-z_0^{-0.5})(1-a-\epsilon)\ln \sqrt{z_0}}\Big)\\
&\quad+O_{6921}\Big(\frac{|t|^4||\phi||^4}{n^{0.1}(1-z_0^{-0.5})^4(1-a-\epsilon)^4(\ln \sqrt{z_0})^4}\Big)+\sum_{i=1}^{n^{0.2}}O_1\Big(\frac{t^2\mathbb{E}|\mathbb{E}[X^2_i-\mathbb{E}X^2_i|\mathcal{F}_i]|}{n}\Big).
\end{align*}

Now, by making use of (\ref{5}), Theorem \ref{correlation}, and the fact that $(\psi_i \circ \sigma^i)_{i \ge 0}$ are reverse martingale differences, we can continue the estimates as
\begin{align*}
&=\exp\Big(-\frac{t^2\Sigma^2}{2}\Big)+O_{6922}\Big(\frac{|t|^4||\phi||^4}{n^{0.1}(1-z_0^{-0.5})^4(1-a-\epsilon)^4(\ln \sqrt{z_0})^4}+\frac{t^2z_0^{-0.5n}||\phi||^2}{(1-z_0^{-0.5})(1-a-\epsilon)\ln \sqrt{z_0}}\Big)\\
&\quad+\sum_{i=1}^{n^{0.2}}O_1\Big(\frac{\sum_{k\in I_{i,1}}t^2\mathbb{E}|\mathbb{E}[\psi_k^2\circ \sigma^k-\mathbb{E}\psi^2_k\circ \sigma^k|\mathcal{F}_i]|}{n}\Big)\\
&=\exp\Big(-\frac{t^2\Sigma^2}{2}\Big)+O_{6922}\Big(\frac{|t|^4||\phi||^4}{n^{0.1}(1-z_0^{-0.5})^4(1-a-\epsilon)^4(\ln \sqrt{z_0})^4}+\frac{t^2z_0^{-0.5n}||\phi||^2}{(1-z_0^{-0.5})(1-a-\epsilon)\ln \sqrt{z_0}}\Big)\\
&\quad +\sum_{i=1}^{n^{0.2}}\sum_{k \in I_{i,1}}O_1\Big(t^2||\psi_k^2||\frac{(\sqrt{z_0})^{-n^{0.2}}}{(1-a-\epsilon)\ln \sqrt{z_0}}\Big)\\
&=\exp\Big(-\frac{t^2\Sigma^2}{2}\Big)+O_{6922}\Big(\frac{|t|^4||\phi||^4}{n^{0.1}(1-z_0^{-0.5})^4(1-a-\epsilon)^4(\ln \sqrt{z_0})^4}+\frac{t^2z_0^{-0.5n}||\phi||^2}{(1-z_0^{-0.5})(1-a-\epsilon)\ln \sqrt{z_0}}\Big)\\
&\quad +O_{64}\Big(\frac{t^2n^3||\phi||^2e^{2||\phi_p||\theta(1-\theta)^{-1}}||\phi_p||^2(\sqrt{z_0})^{-n^{0.2}}}{(1-\theta)^2(1-a-\epsilon)\ln \sqrt{z_0}}\Big),
\end{align*}which concludes a proof.
\end{proof}
\begin{proof}[ A proof of Theorem \ref{deviation}]
We will apply the result of \cite{azuma} to $(\psi_i\circ \sigma^i)_{i \ge 0}$: \[\mu_{\phi_p}\Big(\Big|\frac{\sum_{i \le n-1}\psi_i \circ \sigma^i}{n}\Big|\ge u/2\Big)\le 2\exp\Big(-\frac{u^2n}{8\sup_i||\psi_i \circ \sigma^i||^2_{\infty}}\Big).\] If $n> \frac{2}{u}||\phi||\frac{z_0^{-0.5}}{(1-z_0^{-0.5})(1-a-\epsilon)\ln \sqrt{z_0}}$, then, by (\ref{1}), (\ref{2}) and (\ref{3}),
 we have that

\begin{align*}\mu_{\phi_p}\Big(\Big|\frac{\sum_{i \le n-1}\phi \circ \sigma^i}{n}\Big|\ge u\Big)&\le \mu_{\phi_p}\Big(\Big|\frac{\sum_{i \le n-1}\psi_i \circ \sigma^i}{n}\Big|\ge u/2\Big)+\mu_{\phi_p}\Big(\Big|\frac{H_{n}\circ \sigma^{n}}{n}\Big|\ge u/2\Big)\\
& =\mu_{\phi_p}\Big(\Big|\frac{\sum_{i \le n-1}\psi_i \circ \sigma^i}{n}\Big|\ge u/2\Big)\\
&\le 2\exp\Big(-\frac{u^2n}{8\sup_i||\psi_i \circ \sigma^i||^2_{\infty}}\Big)\\
&\le 2\exp\Big(-\frac{u^2n (1-z_0^{-0.5})^2(1-a-\epsilon)^2(\ln \sqrt{z_0})^2}{72||\phi||^2}\Big).
 \end{align*}
Hence, if $n \ge 1$, then \begin{align*}&\mu_{\phi_p}\Big(\Big|\frac{\sum_{i \le n-1}\phi \circ \sigma^i}{n}\Big|\ge u\Big)\\
&\le 2\exp \Big(\frac{u(1-z_0^{-0.5})(1-a-\epsilon)(\ln \sqrt{z_0})}{36||\phi||z_0^{0.5}}\Big)\exp\Big(-\frac{u^2n (1-z_0^{-0.5})^2(1-a-\epsilon)^2(\ln \sqrt{z_0})^2}{72||\phi||^2}\Big).
\end{align*}
\end{proof}
\begin{proof}[A proof of Theorem \ref{largenumber}]
Take now any $\delta>0$, and choose $u=n^{-(0.5-\delta)}$,
\begin{align*}&\mu_{\phi_p}\Big(\Big|\frac{\sum_{i \le n-1}\phi \circ \sigma^i}{n}\Big|\ge n^{-(0.5-\delta)}\Big)\\
&\le 2\exp \Big(\frac{(1-z_0^{-0.5})(1-a-\epsilon)(\ln \sqrt{z_0})}{36||\phi||z_0^{0.5}}\Big)\exp\Big(-\frac{n^{2\delta} (1-z_0^{-0.5})^2(1-a-\epsilon)^2(\ln \sqrt{z_0})^2}{72||\phi||^2}\Big),\\
&\mu_{\phi_p}\Big(\bigcup_{n\ge N}\Big|\frac{\sum_{i \le n-1}\phi \circ \sigma^i}{n}\Big|\ge n^{-(0.5-\delta)}\Big)\\
&\le 2\exp \Big(\frac{(1-z_0^{-0.5})(1-a-\epsilon)(\ln \sqrt{z_0})}{36||\phi||z_0^{0.5}}\Big)\sum_{n \ge N}\exp\Big(-\frac{n^{2\delta} (1-z_0^{-0.5})^2(1-a-\epsilon)^2(\ln \sqrt{z_0})^2}{72||\phi||^2}\Big).
\end{align*}

Let $N_{x,\delta}:= \sum_{N\ge 0} \mathbbm{1}_{\bigcup_{n\ge N}|\sum_{i \le n-1}\phi \circ \sigma^i(x)|\ge n^{0.5+\delta}}$. Then, by the Borel-Cantelli lemma, for any $n \ge N_{x, \delta}$, \[\Big|\frac{\sum_{i \le n-1}\phi \circ \sigma^i}{n}\Big|\le n^{-(0.5-\delta)},\]
and  $||N_{x,\delta}||_1$
\begin{align*}&\le  2\exp \Big(\frac{(1-z_0^{-0.5})(1-a-\epsilon)(\ln \sqrt{z_0})}{36||\phi||z_0^{0.5}}\Big)\sum_{N\ge 0}\sum_{n \ge N}\exp\Big(-\frac{n^{2\delta} (1-z_0^{-0.5})^2(1-a-\epsilon)^2(\ln \sqrt{z_0})^2}{72||\phi||^2}\Big)\\
&\le 2\exp \Big(\frac{(1-z_0^{-0.5})(1-a-\epsilon)(\ln \sqrt{z_0})}{36||\phi||z_0^{0.5}}\Big)\sum_{n\ge 0}n\exp\Big(-\frac{n^{2\delta} (1-z_0^{-0.5})^2(1-a-\epsilon)^2(\ln \sqrt{z_0})^2}{72||\phi||^2}\Big),
\end{align*}which concludes a proof.
\end{proof}

\section{Applications}\label{app}
In this section  we consider applications of the obtained results to some uniformly hyperbolic linear maps $f\in SL(2d, \mathbb{Z})$ of $2d-$dim flat torus ($d\ge 1$). These maps can be lifted to a topological Markov shift $(\mathbb{A}, \sigma, \mu_{\phi_p})$. We assume also that $f$ is a symmetric matrix with the eigenvalues $\lambda_1\ge \lambda_2 \ge \cdots \ge \lambda_d>1> \lambda_{d+1}\ge \lambda_{d+2} \ge \cdots \ge \lambda_{2d}>0$. 

Thanks to the linearity, the Radon–Nikodym derivative along unstable manifolds (which are in this case linear sub-spaces) $\det D^uf=\exp(-\phi_p)$ is constant, and $||\phi_p||_{\infty}=|\ln \det D^uf|$. Since $f$ is symmetric, then $||\phi_p||_{\infty} =|\ln \det D^uf| =\sum_{i=1}^d\ln \lambda_i$. We choose $||\phi_p||= \max\{1, \sum_{i=1}^d\ln \lambda_i\}$, and take a value of $\theta$ in Definition \ref{shift} equal to $\lambda_d^{-1}$. Now let $A_d=\bigoplus_{i=1}^{d} \begin{pmatrix} 2 & 1 \\ 1 & 1 \end{pmatrix}: \bigoplus_{i=1}^{d}\mathbb{T}^2 \to \bigoplus_{i=1}^{d}\mathbb{T}^2$ where $\bigoplus_{i=1}^{d}\mathbb{T}^2$ is a $2d$-dimensional torus.
\begin{Example}
Let $f=A_dB_dA_d$, where $B_d=\begin{pmatrix} 1 & \textbf{0}& 1 \\\textbf{0} & A_{d-1}& \textbf{0} \\ 1 &\textbf{0} & 2 \end{pmatrix}$. By making use of the argument above, one can find $||\phi_p||$ and $\theta$. The eigenvalues of $f$ are  
\[[(9+6\cos(2\pi j/d)) \pm \sqrt{(9+6\cos(2\pi j/d))^2 -4}]/2,\]
$j=1,2,\cdots,d$. Then $\theta=(9 -\sqrt{77})/2$, $||\phi_p||=d\ln \frac{15+\sqrt{221}}{2}$. By Theorem \ref{clt},

 \[\Big|\mathbb{E}\exp\Big(it \frac{\sum_{0 \le i \le n-1}\phi\circ \sigma^i}{\sqrt{n}}\Big)-\exp\Big(-\frac{t^2\Sigma^2}{2}\Big)\Big|\le C_{t,n}, \]
    where\begin{align*}C_{t,n}=&6922\Big(\frac{|t|^4||\phi||^4}{n^{0.1}(1-z_0^{-0.5})^4(1-a-\epsilon)^4(\ln \sqrt{z_0})^4}+\frac{t^2z_0^{-0.5n}||\phi||^2}{(1-z_0^{-0.5})(1-a-\epsilon)\ln \sqrt{z_0}}\Big)\\
    &+64\Big(\frac{t^2n^3||\phi||^2e^{2d\ln(\frac{15+\sqrt{221}}{2})\frac{9-\sqrt{77}}{\sqrt{77}-7}}d^2\ln^2 \frac{15+\sqrt{221}}{2}(\sqrt{z_0})^{-n^{0.2}}}{(\frac{\sqrt{77}-7}{2})^2(1-a-\epsilon)\ln \sqrt{z_0}}\Big),
    \end{align*} 

and \[\epsilon \in (0, \min\{1-a, \frac{\sqrt{77}-7}{2}\}), \quad a=1-\exp(-d\ln \frac{15+\sqrt{221}}{\sqrt{77}-7}), \]
\[z_0 \in \Big(1, \min\Big\{\Big(\frac{a+\epsilon/2}{a}\Big)^{U}, \frac{9-\sqrt{77}+2\epsilon}{9-\sqrt{77}}, \frac{2-2\epsilon}{9-\sqrt{77}}, e^2\Big\}\Big), \quad U:=\frac{\log (1-\epsilon)}{\log [\frac{9-\sqrt{77}}{2} \epsilon(\frac{\sqrt{77}-7}{2}-\epsilon)]-\log (2||\phi_p||)}\] 

If $d=3$, then $a=0.9999$, $\epsilon=5\times 10^{-5}$, $z_0=1.00000000083$. Therefore, the central limit theorem for $f$ holds, and

\[\Big|\mathbb{E}\exp\Big(it \frac{\sum_{0 \le i \le n-1}\phi\circ \sigma^i}{\sqrt{n}}\Big)-\exp\Big(-\frac{t^2\Sigma^2}{2}\Big)\Big|\le 357.15265\times 10^{56}\times t^4||\phi||^4 n^{-0.1}.\]

By Theorem \ref{deviation}, the large deviation estimate for $f$ is given via the following relation \[\mu_{\phi_p}\Big(\Big|\frac{\sum_{i \le n-1}\phi \circ \sigma^i}{n}\Big|\ge u\Big) \le 2\exp \Big(\frac{u\times 5.76388936\times10^{-16}}{||\phi||}\Big)\exp\Big(-5.9800357\times10^{-30}\times \frac{u^2n }{||\phi||^2}\Big).\]

By Theorem \ref{largenumber} in the law of large numbers we have for any $n \ge N_{x, \delta},$\[\Big|\frac{\sum_{i \le n}\phi \circ f^i(x)}{n}\Big|\le n^{-(0.5-\delta)},\]
where $N_{x,\delta}:= \sum_{N\ge 0} \mathbbm{1}_{\bigcup_{n\ge N}|\sum_{i \le n-1}\phi \circ f^i(x)|\ge n^{0.5+\delta}}$ satisfies the relation \[ ||N_{x,\delta}||_1 \le 2\exp \Big(\frac{u\times 5.76388936\times10^{-16}}{||\phi||}\Big) \sum_{n\ge 1}n\exp\Big(-5.9800357\times10^{-30}\times \frac{n^{2\delta} }{||\phi||^2}\Big)< \infty.\]

\end{Example}
\begin{acknowledgements}
We thank S. Vempala for the construction of the example in section \ref{app} and computing the eigenvalues. LB was partially supported by the NSF grant DMS-2054659.
\end{acknowledgements}

%
%

\medskip

\bibliography{bibtext}

\end{document}